\documentclass[12pt]{amsart}
\usepackage{graphicx}
\usepackage{commath}
\usepackage{stmaryrd}
\usepackage{xcolor}
\newtheorem{lemma}{Lemma}

\newtheorem{theorem}{Theorem}
\theoremstyle{definition}
\newtheorem{definition}{Definition}
\newtheorem{remark}{Remark}
\newtheorem{conjecture}{Conjecture}
\newtheorem{problem}{Problem}

\newtheorem{proposition}{Proposition}

\newcommand{\C}{\mathbb{C}}
\newcommand{\Z}{\mathbb{Z}}

\newcommand{\D}{\mathbb{D}}

\newcommand{\Belt}{\mathrm{Belt}}

\begin{document}
\title{Unique extremality of affine maps on plane domains}
\author{Qiliang Luo \,\,\text{and}\,\, Vladimir Markovi\'c}
\address{\newline Yau's Mathematical Sciences Center  \newline Tsinghua University   \newline Beijing, China}

\today


\begin{abstract} We prove that affine maps are uniquely extremal quasiconformal maps on the complement of a well distribute set in the complex plane answering a conjecture from \cite{markovic}. We construct the required Reich sequence using  Bergman projections, and meromorphic partitions of unity.
\end{abstract}

\maketitle

\section{Introduction}

Let $f:X \to Y$ be a quasiconformal map between Riemann surfaces $X$ and $Y$. By $\Belt(f)=\frac{\overline\partial f}{{\partial} f}$ we denote the Beltrami dilatation of $f$. Then $\Belt(f)$ is a $(-1,1)$ complex form on $X$, while $|\Belt(f)|$ is a measurable function on $X$ such that $||\Belt(f)||_{\infty}<1$. 
We say that $f$  is extremal  if $f$ has the smallest Beltrami dilation in its homotopy class, that is, if $||\Belt(f)||_{\infty}\le ||\Belt(g)||_{\infty}$ for every quasiconformal map $g$ homotopic to $f$. We say that $f$ is uniquely extremal if it is the only extremal map in its homotopy class.

Every quasiconformal map is homotopic to an extremal quasiconformal map (see \cite{ahlfors}).  The classical extremal problem, 
first studied by Gr\"otzsch and Teichm\"uller, is to describe extremal quasiconformal maps, and to decide when they are uniquely extremal. 
When $X$ and $Y$ are finite type Riemann surfaces then every extremal map is of Teichm\"uller type, and it is uniquely extremal. In general, extremal quasiconformal maps may not be of Teichm\"uller type, nor they are necessarily uniquely extremal. 

In \cite{r-s} Reich and Strebel proposed the following special case of the extremal problem which is  general enough, and which should be manageable.

\begin{problem} Characterise discrete subsets $E\subset \C$ such that affine maps are extremal, or uniquely extremal, as  quasiconformal maps on 
$\C \setminus E$. 
\end{problem}
\begin{remark} Let $A_1,A_2:\C\to \C$ be two affine maps neither of which is the identity map. Then $A_1$ is extremal (uniquely extremal) on $\C\setminus E$ if and only $A_2$ is extremal (uniquely extremal) on $\C\setminus E$. 
This follows from the analytical characterisation of extremal maps \cite{hamilton}, \cite{r-s-0}, and uniquely extremal maps \cite{b-l-m-m}.
Thus, the above problem is well posed.
\end{remark}

The methods used to study these problems are partly geometric, and partly analytic, and it is this interplay which makes them interesting. However, even in this generality the problem is very hard. In \cite{r-s} Reich and Strebel provided a somewhat complicated characterisation of discrete sets $E$ such that affine maps are extremal on $\C \setminus E$. On the other hand, they were not aware of a single example  such that affine maps are uniquely extremal on $\C \setminus E$. The first such example was provided in \cite{markovic} where it was shown that affine maps are uniquely extremal on $\C\setminus 
(\Z+i\Z)$ (here  $\Z+i\Z$ denotes the integer lattice in the complex plane). This answered a conjecture by Kra-Reich-Strebel.

\begin{definition}  Let $c>0$. A discrete set $E\subset  \C$ is called c-well distributed if $E$ intersects any  disc of Euclidean radius $c$ in the complex plane. We say that $E$ is well distributed  if it is c-well distributed for some $c>0$.
\end{definition}
The following conjecture was posed in \cite{markovic}.

\begin{conjecture}\label{conj} If $E$ is a well distributed set then affine maps are uniquely extremal on $\C\backslash E$.
\end{conjecture}

Well distributed sets represent a geometric generalization of the integer lattice $\Z+i\Z$, and it is to be expected that the proof that affine maps are uniquely extremal on $\C\setminus (\Z+i\Z)$ should be utilised to prove Conjecture \ref{conj}. However, the method of proof in \cite{markovic} strongly exploits the fact that $\C\setminus (\Z+i\Z)$ is an amenable cover of the square once punctures torus (compare with Mcmullen's proof of the Kra conjecture \cite{mcmullen}), and it can not be extended to the case of well distributed sets. The purpose of this paper is to develop a new approach and prove Conjecture \ref{conj}. In particular, we provide a new and much more conceptual proof of the the main result from \cite{markovic} that affine maps are uniquely extremal on the complement on the integer lattice (the old proof relies heavily on the translation invariance of the integer lattice).

\begin{theorem}\label{thm1} If $E$ is a well distributed set then  affine maps are uniquely extremal on  $\C\backslash E$.
\end{theorem}
\begin{remark} Every $c$-well distributed set can be mapped to a $\frac 18$-well distributed set by a linear isomorphism of $\C$. Since linear maps conjugate affine maps onto affine maps, it suffices to prove  Theorem \ref{thm1} for $\frac 18$-well distributed sets.
\end{remark}

\subsection{Reich sequences} The following theorem  due to Reich \cite{reich}  gives a sufficient condition for an affine map on a plane domain $M\subset \C$ to be uniquely extremal. Let $QD^1(M)$ be the Banach space of all integrable holomorphic quadratic differentials with respect to the $L^1$-norm.
\begin{theorem}\label{ReichCri}
	Let $M$ be a plane domain. An affine map on $M$ is  uniquely extremal  if there exists a sequence $\Phi_n=\phi_n\,dz^2\in QD^1(M)$ satisfying the following three conditions:
	\begin{enumerate}
		\item $\lim\limits_{n \to \infty} \phi_n(z)=1$, \,\,\,\,  $\forall z\in M$,
		\vskip .1cm
		\item $\limsup\limits_{n\to \infty} \int\limits_\C \big(|\phi_n|-Re(\phi_n)\big)|dz|^2<+\infty$,
		\item If $S_{n,K}=\{z\in M:|\phi_n(z)|\geq K\}$, then for each 
		\[
		\liminf_{K\rightarrow\infty}\mathop\int\limits_{S_{K,n}}|\phi_n(z)||dz|^2=0,
		\]
		uniformly in $n\in\mathbb N$.
		
	\end{enumerate}
\end{theorem}

Such a sequence $\Phi_n \subset QD^1(M)$ is called a Reich sequence (we also refer to the corresponding sequence of functions $\phi_n$ as the Reich sequence). It is known to be hard to construct such sequences. We  construct Reich sequences on the complements of  well distributed sets. 

\begin{remark} It follows from \cite{b-l-m-m} that the existence of a Reich  sequence is equivalent to the statement that affine maps are uniquely extremal on $M$.
\end{remark}

\subsection{Quasilattices}  By $L_0$ we denote the integer lattice in the complex plane $\C$. Points in $L_0$ is enumerated by $z_{k,l}=k+li$ where $k,l\in\Z$. 
\begin{definition} A countable subset $L$ in $\C$ is called a quasilattice if its points can be enumerated by $w_{k,l}$, where $k,l\in\Z$, and $w_{k,l}$ satisfies
\[
\big|z_{k,l}-w_{k,l}\big|\leq\frac 18.
\]
\end{definition}
It's easy to observed that a quasilattice is a well distributed set. On the other hand, it will be proved in the following lemma that, any $\frac 18$-well distributed set $E$ will contains a quasilattice $L$ as subset.  Then a Reich sequence for $\C\setminus L$ is also a Reich sequence for $\C\setminus E$, since $QD^1(\C\setminus L)\subset QD^1(\C\setminus E)$. Thus,  it suffices to construct a Reich sequence $\phi_n$ for every plane domain $\C\setminus L$ which is a complement of a quasilattice.
\begin{lemma}\label{qlinE}
Every $\frac 18$-well distributed set $E$ contains a quasilattice $L\subset E$.
\end{lemma}
\begin{proof}
	By $\D_{\frac 18}(z)$ we denote the Euclidean disk of radius $\frac 18$, centered at $z$. Each disk $\D_{\frac 18}(z)$ contains at least one point from $E$. Therefore, we can chose points
	\[
	w_{k,l}\in\D_\frac 18(z_{k,l})\cap E
	\]
	Points $w_{k,l}$ are mutually different since they live in disjoint disks. Moreover, by construction $|w_{k,l}-z_{k,l}|<\frac 18$. Therefore $L=\{w_{k,l}:k,l\in\mathbb Z\}$ is a quasilattice.
\end{proof}

\subsection{Meromorphic partition of unity} 
We let 
\begin{equation}\label{eq-0}
\Omega=\{z\in\C:d(z,L_0)\geq\frac 14\}.
\end{equation}
Note that if $L$ is a quasilattice then $\Omega\subset\C\setminus L$. 

\begin{definition}\label{def-part}
Let $L$ be a quasilattice.  Let $\{P_{p,q}\}$, $p,q\in\Z$,  be a double sequence of meromorphic functions such that each $P_{p,q}\,dz^2 \in QD^1(\C\setminus L)$. We say that $P_{p,q}$ is a meromorphic partition of unity if the following holds:
\begin{enumerate}
	\item 
	There exist constant $C$ such that for any $z\in\Omega$, and any $p,q\in\Z$, we have
	\[
		\bigg|P_{p,q}(z)\bigg|\leq\frac{C}{\exp\left(\frac{|z-z_{p,q}|}{C}\right)}.
		\]
		\item 
		$$
		\sum_{k,l}P_{k,l}(z)=1,\,\,\,\,\,\,\,\, z\in \C\setminus L.
		$$ 
	\end{enumerate}
\end{definition}
\begin{remark} It follows from the the estimate (1)  that the sequence in (2) absolutely converges for each $z\in \C\setminus L$.
\end{remark}
Constructing such a  meromorphic partition of unity is a key idea behind constructing  a Reich sequence as illustrated by the following theorem.
\begin{theorem}\label{thm3}
	Assume $L$ is a quasilattice equipped with a meromorphic partition of unity $P_{p,q}$, $p,q\in \Z$. Set
	\[
	\phi_n(z)=\sum\limits_{k,l} \frac{P_{k,l}(z)}{\big(|\frac{z_{k,l}}{n}|+1\big)^4 }
	\]
	for $z \in \C\setminus L$.	 Then $\phi_n$ is a Reich sequence.
\end{theorem}

\subsection{Organization} In the next two section we construct a meromorphic partition of unity $P_{p,q}$ on $\C\setminus L$. In Section \ref{sec-berg} we recall the notion of the Bergman projection. The main result of this section is the estimate in Lemma \ref{bergest}.
In Section \ref{sec-berg-1} we define the double sequence $P_{p,q}$ using the Bergman  projections of characteristic functions of squares centred at the lattice points $z_{p,q}$. Theorem \ref{thm3} is proved in Section \ref{sec-reich} by elementary computations.

\section{The Bergman kernel}\label{sec-berg}
As preliminary, we recall the definition of the Bergman kernel for any hyperbolic Riemann surface and state some of its well known  properties (see Section 12 in \cite{f-m} for example). 
The Bergman kernel function on the unit disc $\D$ is given by
\[
K(z,w)=\frac{1}{(1-z\overline w)^4}.
\]
The Bergman kernel $B_\D$ is the differential form given by 
\[
B_\D(z,w)=K(z,w)\, dz^2d\overline w^2.
\]
By direct computation one finds that  the Bergman kernel $B_\D$ is M\"obius transformation invariant, that is, for $A\in Aut(\D)$ we have
\begin{equation}\label{eq-inv}
K(Az,Aw)A'(z)^2\overline{A'(w)}^2=K(z,w).
\end{equation}
Another key property of this kernel is the following integral identity (also proved by elementary computation). Namely, let $W\subset \D$, and let $w$ be the point such that $A(w)=0$. Then
\begin{equation}\label{eq-inv-1}
\rho^2_\D(w)\int\limits_{A(W)} \, |dz|^2=\int\limits_{W} |K(z,w)|\, |dz|^2,
\end{equation}
where
$$
\rho_\D(w)=\frac{1}{(1-|w|^2)}
$$
is the density of the hyperbolic metric on $\D$.

\subsection{The Bergman projection} A word on notation first. Suppose $\omega$ is a volume form on $X\times X$ obtained as a product of two volume forms on $X$. Let $(z,w)$ denote local coordinates in $X\times X$. If $Y\subset X$, then the integral
$$
\int\limits_{Y_{z}} \omega
$$ 
is a volume form on $X$ (with local coordinate $w$), that is the notation $Y_z$ means that we are integrating with respect to the $z$ variable. And likewise for $Y_w$. Also, if we fix a point $p\in X$ then $\omega(p,w)$ is volume form with respect to $w\in X$ (this is the evaluation of $\omega$ at $(p,w)$). And likewise for  $\omega(z,p)$.

To each measurable quadratic differential $f$ on $\D$, we associate the holomorphic quadratic differential $f*B_\D$ on $\D$ given by
\begin{equation}\label{Bergmanconv-0}
f*B_\D=\frac{3}{2\pi}\int\limits _{\D_{w}} f(w) dV^{-1}_\D(w) B_\D(z,w),
\end{equation}
where the tensor $dV^{-1}_\D$ is given by
$$
dV^{-1}_\D(w)=\frac{1}{\rho^{2}_\D(w)|dw|^2}.
$$
This convolution formula is well defined when $f$ has a finite $L^1$
norm in which case $f*B_\D\in QD^1(\D)$,  or if $f$ has a finite Bers norm in which case $f*B_\D\in QD^\infty(\D)$, where 
$\Phi=\phi\, dw^2\in QD^\infty(\D)$ if
$$
\sup_{w\in\D} \rho^{-2}(w)|\phi(w)|<\infty.
$$
 If $f\in QD^1(\D)$, or $f\in QD^\infty(\D)$, then $f*B_\D\equiv f$.

Let $S$ be a hyperbolic Riemann surface given as the quotient $S=\mathbb D / \Gamma$, where $\Gamma$ is a Fuchsian group. We use $z$ and $w$ to denote the local coordinates on both $\D$ and $S$.  
The differential form $B_S$, called the Bergman kernel on the surface $S$, is defined by Poincare series
\[
B_S(z,w)=\sum_{A\in\Gamma}K(Az,w)A'(z)^2.
\]
The Poincare series is  absolutely convergent since for a fixed $w\in\D$ the integral of $|K|$ over $z\in\D$  is bounded. 
For any measurable quadratic differential $f$ on  $S$ we define the Bergman projection by
\begin{equation}\label{Bergmanconv}
f*B_S=\frac{3}{2\pi}\int\limits _{S_{w}} f(w) dV^{-1}_S(w) B_S(z,w),
\end{equation}
with the tensor
$$
dV^{-1}_S(w)=\frac{1}{dV_S(w)},
$$
where $dV_S(w)=\rho^2_S(w)|dw|^2$ is the volume form of the hyperbolic metric on $S$ (here $\rho_S$ is the density of the hyperbolic metric on $S$). This convolution formula is well defined when $f$ has a finite $L^1$
norm in which case $f*B_S\in QD^1(S)$,  or if $f$ has a finite Bers norm in which case $f*B_S\in QD^\infty(S)$. If $f\in QD^1(S)$, or $f\in QD^\infty(S)$, then $f*B_S\equiv f$. The following lemma is our main estimate about the Bergman kernel. 
\par 
\begin{lemma}\label{bergest}
	Let $S$ denote a hyperbolic Riemann surface. Given  any subset $U\subset S$, and $p \in S$, we have
	\[
	\int\limits_{U_{z}}|B_S(z,p)|  \leq \frac{4\pi }{\exp\big(d_S(U,p)\big)}dV_S(p),
	\]
	where $d_S(U,p)$ denotes the hyperbolic distance between $U$ and $p$.
		\end{lemma}
	\begin{remark} The left hand side in the above inequality is a volume form on $S$ (in the coordinate $p$). Thus, the inequality represents a pointwise comparison between two volume forms on $S$.
	\end{remark}
\begin{proof}
Let $V\subset \D$ be a  fundamental domain for $U\subset S$. Let  $w_p \in \D$ denote a lift of $p$. Then 
\begin{align}\label{M1}
	\int\limits_{U_{z}}|B_S(z,p)| &=\left( \int\limits_{V}\bigg|\sum_{A\in\Gamma} K(Az,w_p)A'(z)^2 \bigg|\,|dz|^2\right)|dw|^2 \notag \\ &\leq \left( \sum_{A\in\Gamma}\int\limits_{A(V)}|K(z,w_p)|\, |dz|^2\right)|dw|^2\\
	&=\left(\int\limits_{\Gamma(V)}|K(z,w_p)|\,|dz|^2\right)|dw|^2 \notag.
\end{align}	

	Let $T\in Aut(\D)$ be the M\"obius transformation which maps $w_p$ to $0$. By replacing $W=\Gamma(V)$ in the invariance formula (\ref{eq-inv-1}),  we obtain
	\begin{equation}\label{M2}
			\int\limits_{\Gamma(V)}|K(z,w_p)|\, |dz|^2=\rho^2_\D(w_p)\int\limits_{T(\Gamma(V))} \, |dz|^2.
	\end{equation}
	
	Now, we have  $d_S(U,p)=d_\D(\Gamma(V),w)=
	d_\D(T\big(\Gamma(V)),0\big)$. 
	Thus, $T\big(\Gamma(V)\big)$ is contained in the set
	\[\left\{\xi \in\D: |\xi| \geq \tanh\left(\frac {d_S(U,p)}{2}\right)\right\}.\]
	We find
	\begin{equation*}
		\begin{aligned}
		\int\limits_{T(\Gamma(V))} \,|dz|^2 \leq \int\limits_{|\xi|\geq \tanh\big(\frac {d_S(U,p)}{2}\big)}\, |d\xi|^2=
		\frac{4\pi}{\cosh^2\big(\frac {d_S(U,p)}{2}\big)}\leq\frac{4\pi}{\exp\big(d_S(U,p)\big)}.
		\end{aligned}
	\end{equation*}
	Combining this with the equations (\ref{M1}) and (\ref{M2}) completes the proof.
\end{proof}

\section{The Bergman projection and the meromorphic partition of unity}\label{sec-berg-1}

In this  section we  construct an explicit meromorphic partition of unity for any quasilattice $L$. By $B_L$ we denote the Bergman kernel for the Riemann surface $\C\setminus L$.  To each measurable quadratic differential $f$ on $\C\setminus L$, we associate the Bergman  projection  defined by
\[
f*B_L=\frac{3}{2\pi}\int\limits_{(\C\setminus L)_{w}} f(w)dV^{-1}_L(w)B_L(z,w).
\]
In the next subsection we prove:
 \begin{lemma}\label{Pestimate} There exist a constant $C$ with the following properties.  Set $W=[-\frac 12,\frac 12]\times [-\frac 12,\frac 12]$, and consider the measurable quadratic differential $f(z)=\chi(z)\,dz^2$, where $\chi$ is the characteristic function of $W$. Then for any quasilattice $L$, and any $z_0\in \Omega$, we have 
	\[
	\left|\frac{(f*B_L)}{|dz|^2}\right|(z_0)\leq\frac{C}{\exp\big(\frac{|z_{0}|}{C}\big)}.
	\]
\end{lemma}

\begin{remark} Note that $|(f*B_L)|$ is a volume form on $\C\setminus L$ since $f*B_L$ is a quadratic differential. The right hand side in the above inequality is a function  on $\C\setminus L$. Thus, the lemma gives a pointwise comparison between two functions on $\C\setminus L$.
\end{remark}
 The construction of meromorphic partition of unity is as follows. Let $W_{k,l}=z_{k,l}+W$ be the subsets of $\C$. The complex plane $\C$ is a disjoint  union of the sets $W_{k,l}$. By $\chi_{k,l}(z)$ we denote the characteristic function of $W_{k,l}$, and $f_{k,l}=\chi_{k,l}\,dz^2$. For any quasilattice $L$, we can project the integrable quadratic differentials $f_{k,l}$ to the integrable holomorphic quadratic differentials $f_{k,l}*B_L \in QD^1(\C\setminus L)$. Set $P_{k,l}\,dz^2=f_{k,l}*B_L$. We claim that the double sequence $P_{k,l}$ is a meromorphic partition of unity. 
 
 Since 
 \[\sum_{k,l}\chi_{k,l}\equiv  1,\]
by the reproducing property of the Bergman projection we have
 \[
 \sum_{k,l} P_{k,l}\equiv 1. 
 \] 
Thus $P_{k,l}$ satisfies the second condition from Definition \ref{def-part}. It follows from Lemma \ref{Pestimate} that $P_{k,l}$  satisfies the first condition.  To complete the construction of the meromorphic partition of unity, it remains to prove Lemma \ref{Pestimate}. This is done in the next section.

\subsection{Proof of Lemma \ref{Pestimate}}

In the remainder of the proof we let
$$
h=\frac{f*B_L}{dz^2}.
$$
Then $h$ is a holomorphic function on $\C\setminus L$. To prove the lemma we need to show
\begin{equation}\label{eq-ope}
		|h(z_0)| \le \frac{C}{\exp\big(\frac{|z_{0}|}{C}\big)},
\end{equation}
for every $z_0\in \Omega$.

By the construction of $\Omega$, there exists a constant $t_0>0$ such that the injectivity radius (with respect to the hyperbolic metric on $\C\setminus L$ is at greater than $t_0$ at all points in $\Omega$. Let $\xi \in \Omega$, and choose a uniformizing map $\pi:\D\to \C\setminus L$ such that $\pi(0)=\xi$. Then the restriction of $\pi$ to the disc $\D_{r_{0}}(0)$ is univalent where $r_0$ is a function of $t_0$.

Applying the Koebe $1/4$ Theorem to the restriction of $\pi$ to $\D_{r_{0}}(0)$ implies that the disc $\D_{r_{1}}(\xi)\subset \C\setminus L$,
where 
$$
r_1=\frac{r_0}{4}|\pi'(0)|.
$$
But, by construction we know that $r_1<2$. This yields the estimate
 $$
 |\pi'(0)|<\frac{8}{r_{0}}, 
 $$
which in turn implies
$$
s_0< \frac{ds_{L}}{|dz|}(\xi) ,
$$
for every $\xi\in \Omega$, where $s_0=\frac{r_0}{8}$.  By $ds_L$ we denote the hyperbolic length element on $\C\setminus L$, and by $d_L$ the hyperbolic distance on $\C\setminus L$. 
This yields the distance estimate:
\begin{proposition} Let $z,w\in \C \setminus L$. Then
\begin{equation}\label{eq-j2} 
d_L(z,w)\ge \frac{s_0}{2}|z-w|-2s_0.
\end{equation}
\end{proposition}
\begin{proof} Let $\alpha\subset \C\setminus L$ be a smooth arc connecting $z$ and $w$, and set $\beta=\alpha\cap \Omega$. Then
the Euclidean length of $\beta$ is greater than $(1/2)|z-w|-2$. This implies
$$
d_L(z,w)\ge s_0(\frac{1}{2}|z-w|-2)>\frac{s_0}{2}|z-w|-2s_0.
$$
\end{proof}

The proof of Lemma \ref{Pestimate} is based on the equation (\ref{eq-j2}), and Lemma \ref{bergest}. 
Let $z_0\in \Omega$, and let $D$ denote the disc of radius $\frac 18$ centred at $z_0$. Then $D\subset \C\setminus L$. From the Mean Value Theorem we derive 
\[
h(z_0)=\frac{32}{\pi}\int\limits_{D}h(\xi)\,|d\xi|^2.
\]
Then
\begin{equation}\label{eq1010}
	\begin{aligned}
		|h(z_0)|&\leq \frac{32}{\pi}\int\limits_{D_{z}}\int\limits_{\C_w} \left| f(w)dV^{-1}_L(w) B_L(z,w)\right|\\
		&=\frac{32}{\pi}\int\limits_{W_w}\left( |dw|^2 dV^{-1}_L(w)\int\limits_{D_{z}}|B_L(z,w)|\right).	
				\end{aligned}
\end{equation}
On the other hand, applying Lemma \ref{bergest} yields
\begin{equation}\label{eq-lab}
\int\limits_{D_{z}}|B_L(z,w)|\le \frac{4\pi }{\exp(d_L(D,w))}\, dV_L(w).
\end{equation}
Replacing this in (\ref{eq1010}) we get 
\begin{align*}
		|h(z_0)| &\leq \frac{32}{\pi}\int\limits_{W_w}\left(\frac{4\pi }{\exp(d_L(D,w))} \right)\,|dw|^2 \\
		&\le 	 \frac{16}{\exp(d_L(D,W))} \le \frac{16}{\exp\big(\frac{s_0}{2}|z_0|-2s_0-2 \big)}, 
		\end{align*}
where in the last inequality we used (\ref{eq-j2}). Letting
$$
C=\max\{16\exp(2s_0+2), \frac{2}{s_{0}} \}
$$
proves (\ref{eq-ope}) and  the lemma.

\section{The Reich sequence $\phi_n $}\label{sec-reich}
Let $\alpha(n)=\{\alpha_{p,q}(n)\}$ be the double sequence of complex numbers given by
\[
\alpha_{p,q}(n)=\frac{1}{\left(\frac{|z_{p,q}|}{n}+1\right)^4}.
\]
Consider a quasilattice $L$ equipped with a fixed meromorphic partition of unity $P_{p,q}$, $p,q\in\Z$. Set
\[
\phi_n=\sum_{k,l}\alpha_{k,l}(n)P_{k,l}.
\]
It follows from  the inequality (1) in Definition \ref{def-part} that the $L^1$-norms of the quadratic differentials  $P_{k,l}$ are uniformly bounded (uniform in $k,l$).  Since $\sum_{k,l}\alpha_{k,l}(n)<\infty$ for every $n$, it is clear that $\phi_n(z)\subset QD^1(\C\setminus L)$. In this section, we prove the Theorem \ref{thm3} which states that $\phi_n$ is a Reich sequence.

\subsection{Properties of $\phi_n$}
In this subsection we prove two preliminary lemmas providing  preliminary information about $\phi_n$. For any $p,q,k,l,n$, we have
	\[
	\left|\alpha_{p,q}(n)-\alpha_{k,l}(n)\right|=\alpha_{p,q}(n)\alpha_{k,l}(n)\left|\left(\frac{|z_{p,q}|}{n}+1\right)^4-\left(\frac{|z_{k,l}|}{n}+1\right)^4\right|.
	\]
	We apply the Lagrange theorem to obtain
	\[
	\left|\left(\frac{|z_{p,q}|}{n}+1\right)^4-\left(\frac{|z_{k,l}|}{n}+1\right)^4\right|\leq \frac{4|z_{p,q}-z_{k,l}|}{n}\left|\frac{\max\{|z_{k,l}|,|z_{p,q}|\}}{n}+1\right|^3.
	\]
	Since
	\begin{equation*}
	\begin{aligned}
		\alpha_{k,l}(n)\left|\frac{\max\{|z_{k,l}|,|z_{p,q}|\}}{n}+1\right|^3
&=\frac{1}{\frac{|z_{k,l}|}{n}+1}\left(\frac{\max\{|z_{k,l}|,|z_{p,q}|\}+n}{|z_{k,l}|+n}\right)^3\\
&\leq \left(\frac{|z_{k,l}|+|z_{p,q}|+1}{|z_{k,l}|+1}\right)^3=\left(2+\frac{|z_{p,q}|-|z_{k,l}|-1}{|z_{k,l}|+1}\right)^3\\
&\leq \left(2+|z_{k,l}-z_{p,q}|\right)^3,
	\end{aligned}	
	\end{equation*}
	we derive the estimate
	\begin{equation}\label{alpha}
		|\alpha_{p,q}(n)-\alpha_{k,l}(n)|\leq\frac{4\alpha_{p,q}(n)}{n}(2+|z_{k,l}-z_{p,q}|)^4.
	\end{equation}

\begin{lemma}\label{mainesti}
	There exist a constant $C$ such that for any $n,p,q$, and any $z\in W_{p,q}\, \cap\, \Omega$, we have
	\[
	|\phi_n(z)-\alpha_{p,q}(n)|\leq C\frac{\alpha_{p,q}(n)}{n}.
	\]
\end{lemma}
\begin{proof}
	Let $C_1$ be the constant from the inequality (1) in Definition \ref{def-part}. Using the identity (2) from Definition \ref{def-part}, we have 
	\begin{equation}\label{Idk}
		\begin{aligned}
			\left|\phi_n(z)-\alpha_{p,q}(n)\right|&=\left|\sum_{k,l}\alpha_{k,l}(n)P_{k,l}(z)-\alpha_{p,q}(n)\sum_{k,l}P_{k,l}(z)\right|\\
			&\leq\sum_{k,l}\left|\alpha_{k,l}(n)-\alpha_{p,q}(n)\right|\text{ }|P_{k,l}(z)|\\
			&\leq C_1\sum_{k,l}\frac{|\alpha_{k,l}(n)-\alpha_{p,q}(n)|}{\exp(|z-z_{k,l}|/C_1)}
		\end{aligned}
	\end{equation}
	for every $n,p,q$, and $z\in W_{p,q}\, \cap\, \Omega$. Since $z\in W_{p,q}$, we have 
	\[
	|z-z_{k,l}|\geq |z_{p,q}-z_{k,l}|-\frac{\sqrt{2}}{2},
	\]
	so 
	\[
	\frac{1}{\exp(|z-z_{k,l}|/C_1)}\leq \exp\left(\frac{\sqrt{2}}{2C_1}\right)\frac{1}{\exp(|z_{k,l}-z_{p,q}|/C_1)}
	\]
	Combining this with the equations (\ref{alpha}) and (\ref{Idk}), we obtain
	\[
	|\phi_n(z)-\alpha_{p,q}(n)|\leq \left(C_1\sum_{k,l}\frac{4(2+|z_{k,l}-z_{p,q}|)^4\exp\left(\frac{\sqrt{2}}{2C_1}\right)}{exp\left(\frac{|z_{k,l}-z_{p,q}|}{C_1}\right)}\right)\frac{\alpha_{p,q}(n)}{n}.
	\]
This proves the lemma by letting	
$$
C=C_1\sum_{k,l}\frac{4(2+|z_{k,l}-z_{p,q}|)^4\exp\left(\frac{\sqrt{2}}{2C_1}\right)}{exp\left(\frac{|z_{k,l}-z_{p,q}|}{C_1}\right)}
$$
(note that $C$ does not depend on $p$ and $q$).	\end{proof}

\begin{lemma}\label{lemma-moja} For every $z\in W_{p,q}\cap\Omega$, the inequality
\[
|\phi_n(z)|-Re(\phi_n(z))\leq\frac{C^2}{n^2}\alpha_{p,q}(n),
\]
holds for $n$ large enough, where $C$ is the constant from Lemma \ref{mainesti}.
\end{lemma}
\begin{proof} Note that if a complex number $\lambda$ satisfies $|\lambda-1|\leq\epsilon\leq \frac 12$ then
\[
|\lambda|-Re(\lambda)=\frac{Im(\lambda)^2}{|\lambda|+Re(\lambda)}\leq\epsilon^2.
\]
Combining this with  Lemma \ref{mainesti}  proves the  lemma. \end{proof}

\subsection{Proof of the Theorem \ref{thm3}}
In this subsection we  prove that $\phi_n$ is a Reich sequence.  We show that $\phi_n$ has the following three properties which ensures it is a Reich sequence:
	\begin{enumerate}
		\item $\lim_{n\rightarrow\infty}\phi_n(z)=1$ , for any $z\in \C\setminus L$.
		\item $\mathop\int\limits_\C \Big(|\phi_n|-Re\left(\phi_n\right)\Big)|dz|^2$ is uniformly bounded.
		\item If $S_{n,K}=\{z\in \C:|\phi_n(z)|\geq K\}$ then we have
		\[
		\lim_{K\rightarrow\infty}\mathop\int\limits_{S_{K,n}}|\phi_n(z)||dz|^2=0
		\]
		uniformly in $n\in\mathbb N$.
	\end{enumerate}
The complex plane $\C$ is decomposed into the domain $\Omega$, and the radius $\frac 14$ disks  centred at the points $z_{k,l}$. 
Lemma \ref{mainesti} and Lemma \ref{lemma-moja} holds for points in $\Omega$ which enables us to prove the above three properties when  $\phi_n$ is restricted to $\Omega$. That this is sufficient follows from the following lemma proved in Proposition 3.1 and  Lemma 4.1 in \cite{markovic}.  
By  $\D_{\frac 12}$ we denote the disk of radius $\frac{1}{2}$ centred at $0$.
\begin{lemma}\label{localesti} There exist a constant $C$ with the following properties. Suppose $f$ is a meromorphic function  on $\overline{\D_{\frac 12}}$, which   is holomorphic on $\D_{\frac 12}(0)\backslash\{0\}$, and has the first order pole at $0$. Assume that for every $z\in\partial\D_{\frac 12}$, we have
	\[
	|f(z)-1|\leq\epsilon\leq\frac 12.
	\]
	Then 
	\begin{equation}\label{eq-1}
		\int\limits_{\D_{\frac 12}}\big(|f|-Re(f)\big)|dz|^2\leq C\epsilon^2.		
	\end{equation}
Let $S_K=\{z\in\D_{\frac 12}:|f(z)|\geq K\}$. Then for any $K\geq 100$ we have
\begin{equation}\label{eq-2}
\int\limits_{S_K}|f||dz|^2\leq\frac{C\epsilon^2}{K}.
\end{equation}
\end{lemma}

Now, consider $z\in\Omega$. Then  Lemma \ref{mainesti} enables us to compute the limit
$$
\lim_{n\rightarrow\infty}\phi_n(z)=1. 
$$
From (\ref{eq-1}) in Lemma \ref{localesti}, and the mean value theorem for holomorphic functions, one easily concludes that the same holds for every $z\in \C\setminus L$.
Thus, we have verified the first property of the sequence $\phi_n$. It remains to do the same with the second and the third property.

Let $C_1$ be the constant from  Lemma \ref{mainesti}, and $C_2$ the constant from  Lemma \ref{localesti}. We apply  Lemma \ref{mainesti} to the function
\[
f_{n,k,l}(z)=\frac{\phi_n(z+w_{k,l})}{\alpha_{k,l}(n)},
\]
and conclude that for any $n,k,l$, and any $z\in \partial\D_{\frac 12}$, we have
\[
|f_{n,k,l}(z)-1|\leq\frac {C_1} {n}.
\]
Thus, $f_{n,k,l}$ satisfies assumption from Lemma \ref{localesti} for sufficiently large $n$. This implies
\begin{equation}\label{eq-11}
\int_{\D_{\frac 12}(w_{k,l})}\big(|\phi_n|-Re(\phi_n)\big)|dz|^2\leq\frac {C_2C_1^2}{n^2}\alpha_{k,l}(n),
\end{equation}
and	
\begin{equation}\label{eq-22}	
\int\limits_{S_{n,k,l,K}}|\phi_n||dz|^2\leq \frac{C_2C_1^2}{n^2K}\alpha^2_{k,l}(n),
\end{equation}
where $S_{n,k,l,K}=\{z\in\D_{\frac 12}(w_{k,l}):|\phi_n(z)|\geq K\}$. 

Now we prove the second property of $\phi_n$. Combining the pointwise estimate on $W_{p,q}\cap\Omega$ from Lemma \ref{lemma-moja}, with (\ref{eq-11}), we obtain
\begin{equation*}
	\begin{aligned}
		\int\limits_{W_{k,l}}\big(|\phi_n|-Re(\phi_n)\big)|dz|^2&\leq \int\limits_{W_{k,l}\cap\,\Omega}\big(|\phi_n|-Re(\phi_n)\big)|dz|^2+\int\limits_{\D_{\frac 12}(w_{k,l})}\big(|\phi_n|-Re(\phi_n)\big)|dz|^2\\
		&\leq \frac{(1+C_2)C_1^2}{n^2}\alpha_{k,l}(n),
	\end{aligned}
\end{equation*}
and therefore
\begin{equation*}
\begin{aligned}
	\overline\lim_{n\rightarrow \infty}\int\limits_\C \big(|\phi_n|-Re(\phi_n)\big)|dz|^2&=\overline\lim_{n\rightarrow\infty}\sum_{k,l}\int\limits_{W_{k,l}}\big(|\phi_n|-Re(\phi_n)\big)|dz|^2\\
	&\leq(1+C_2)C_1^2\, \overline\lim_{n\rightarrow \infty}\sum_{k,l}\frac{1}{n^2}\frac{1}{\left(\left|\frac{z_{k,l}}{n}\right|+1\right)^4}\\
	&=(1+C_2)C_1^2\int\limits_{\C}\frac{1}{(|z|+1)^4}|dz|^2<\infty
\end{aligned}
\end{equation*}
confirming that $\phi_n$ has the second property.

Finally we prove thr $\phi_n$ satisfies the third condition 
	Let $S_{n,K}=\{z\in \C:|\phi_n(z)|\geq K\}$. If $K\geq 1+C_1$ then from  Lemma \ref{mainesti} we conclude that $S_{n,K}$ is disjoint from $\Omega$, and thus
\[
S_{n,K}=\bigcup_{k,l}S_{n,k,l,K}.
\]
We apply the second estimate in Lemma \ref{localesti} to obtain
\begin{equation*}
	\begin{aligned}
		\overline\lim_{n\rightarrow\infty}\int\limits_{S_{n,K}}|\phi_n||dz|^2&\leq\overline\lim_{n\rightarrow\infty}\sum_{k,l}\int\limits_{S_{n,k,l,K}}|\phi_n||dz|^2\\
		&\leq \frac{C_2C_1^2}{K}\overline\lim_{n\rightarrow\infty}\frac{1}{n^2}\frac{1}{\left(\left|\frac{z_{k,l}}{n}\right|+1\right)^8}\\
	&=\frac{C_2C_1^2}{K}\int\limits_{\C}\frac{1}{(|z|+1)^8}|dz|^2=O\left(\frac 1K\right)
	\end{aligned}
\end{equation*}
This complete the proof of the Theorem \ref{thm3}.

\end{document}